\documentclass [12pt]{article}
\usepackage{amssymb,amsmath,amsthm}
\usepackage{color}

\DeclareMathOperator{\ind}{ind}

\DeclareMathOperator{\dom}{dom}

\usepackage[cp1250]{inputenc}
\usepackage[active]{srcltx}

\newtheorem{thm}{Theorem}[section]

\newtheorem{lem}[thm]{Lemma}

\newtheorem{ex}[thm]{Example}
\DeclareMathOperator{\supp}{supp}
\newcommand{\C}{\mathbb{C}}
\newcommand{\R}{\mathbb{R}}

\newcommand{\Z}{\mathbb{Z}}
\newcommand{\la}{\lambda}

\long\def\comment#1{{}}

 \title{Closable Hankel operators and moment problems}
 \author{Christian Berg and Ryszard Szwarc}

\begin{document}

 \maketitle

\begin{abstract} In a paper from 2016 D. R. Yafaev considers Hankel operators associated with  
Hamburger moment sequences $q_n$ and claims that the corresponding Hankel form is closable if and only if the moment sequence tends to 0. The claim is not correct, since we prove closability for any indeterminate moment sequence but also for certain determinate moment sequences corresponding to measures with finite index of determinacy. It is also established that Yafaev's result holds if the moments satisfy $\root{2n}\of{q_{2n}}=o(n)$.  
\end{abstract}


\section{Introduction}

In \cite{Y} Yafaev considers Hankel operators associated with  Hamburger moment sequences
\begin{equation}\label{eq:mom}
q_n=\int_{-\infty}^\infty x^n\,dM(x), \quad n=0,1,\ldots,
\end{equation}
where $M$ is a positive measure on the real line such that the set of polynomials $\C[x]$ is contained in the Hilbert space $L^2(M)$.

  We use the notation of \cite{Y} and let $\mathcal D$ denote the dense subspace of $\ell^2=\ell^2(\Z_+)$ of complex sequences with only finitely many non-zero terms. 
The standard orthonormal basis in $\ell^2$ is denoted $e_n, n=0,1,\ldots$.

Furthermore, we let $A: \mathcal D\to \C[x]$ denote the operator
\begin{equation}\label{eq:pol}
Ag(x)= \sum_{n\ge 0} g_nx^n,\quad g=(g_0,g_1,\ldots)\in\mathcal D,
\end{equation}
considered as a densely defined operator from the Hilbert space $\ell^2$ to $L^2(M)$. 

The Hankel form $q[g,g]$ defined on $\mathcal D$ by
\begin{equation}\label{eq:han}
q[g,g]:=\sum_{n,m\ge 0} q_{n+m}g_n\overline{g_m},\quad g\in\mathcal D
\end{equation}
clearly satisfies
\begin{equation}\label{eq:fu}
q[g,g]=||Ag||^2_{L^2(M)},
\end{equation}
which gives the following result, see \cite[Lemma 2.1]{Y}.

\begin{lem} The form $q[g,g]$ is closable in  $\ell^2$ if and only if $A$ is closable.
\end{lem}  

Because of this result we shall only consider closability of $A$ and leave aside closability of the form  $q$. 

The main result \cite[Theorem 1.2]{Y} can be stated like this.

\begin{thm}\label{thm:Yafaev} Let $q_n$ denote the moments \eqref{eq:mom}. Then the following conditions are equivalent:
\begin{enumerate}
\item [(i)] The operator $A$ in \eqref{eq:pol} is closable.
\item[(ii)] $\lim_{n\to\infty} q_n=0$.
\item[(iii)] The measure satisfies $M(\R\setminus (-1,1))=0$, in other words $ \supp(M)\subseteq[-1,1]$ and $M(\{\pm 1\})=0$.
\end{enumerate}
\end{thm}

It is elementary that (ii) and (iii) are equivalent  and that these conditions imply that  (i) holds. However, (i) does not imply (ii). We shall come back to where the proof in \cite{Y} breaks down, but start by giving  our main results:

\begin{thm}\label{thm:indet} 
If the measure $M$ is indeterminate, then $A$ is closable.
\end{thm}

\begin{thm}\label{thm:det} 
There exist determinate measures with unbounded support such that $A$ is closable. This holds in particular for all determinate measures with finite index of determinacy.
\end{thm}

\begin{thm}\label{thm:strongdet} Suppose the moments satisfy $\root{2n}\of{q_{2n}}=o(n)$.
Then the moment problem is determinate and if $A$ is closable, then condition (iii) holds.
\end{thm}

The proof of the last theorem follows the proof of Yafaev, but as the first two theorems show, some kind of "strong" determinacy condition is necessary for $(i)\implies (iii)$ to hold. We do not know if the condition of Theorem~\ref{thm:strongdet} is optimal.

Let us give some background material for these theorems, see \cite{Ak} for details.
Associated with the moments \eqref{eq:mom} we have
the orthonormal polynomials $(P_n)$, which are  uniquely determined by the conditions
\begin{equation}\label{eq:orpo}
\int_{-\infty}^\infty P_n(x)P_m(x)\,dM(x)=\delta_{n,m},
\end{equation}
when we assume that all $P_n$ have  positive leading coefficients.

If the moment problem is indeterminate, there exists an infinite convex set $V$ of measures $M$ satisfying \eqref{eq:mom}. All measures $M\in V$ have unbounded support. Among the solutions are the Nevalinna extremal or in short the  N-extremal, which are precisely the measures $M\in V$ for which $\C[x]$ is dense in $L^2(M)$ by a theorem of M. Riesz. The N-extremal measures are discrete measures supported by the zero set $\Lambda$ of certain entire functions of minimal exponential type, i.e., of the form
$$
M=\sum_{\la\in\Lambda} c_\la \delta_{\la}, \quad c_{\la}>0.
$$ 

By a theorem going back to Stieltjes in special cases, the following remarkable fact holds: If one mass is removed from  $M$, then the new measure becomes determinate, i.e., 
$$
\widetilde{M}:=M-c_{\la_0}\delta_{\la_0},\quad \la_0\in\Lambda
$$ 
is determinate. For details see e.g. \cite{B:C}, where this result was exploited. The measure
$\widetilde{M}$ is a so-called determinate measure of index of determinacy 0 and if further $n\ge 1$ masses are removed we arrive at a determinate measure $M'$ of index of determinacy $n$, in symbols $\ind(M')=n$. See \cite{B:D}, which contains an intrinsic definition of such measures $M$ by the study of an index $\ind_z(M)$ associated to a point $z\in\C$. Finally, in \cite[Equation (1.5)]{B:D1} we define $\ind(M):=\ind_z(M)$ for $z\in\C \setminus\supp(M)$ because $\ind_z(M)$ is independent of $z$ outside the support of $M$. 

In  the indeterminate case  the
polynomials $P_n$ form  an orthonormal basis in $L^2(M)$ for all the N-extremal solutions $M$, and for the other solutions $M$ they form an orthonormal basis in the closure $\overline{\C[x]}^{L^2(M)}$.

It is known that this closure is isomorphic as Hilbert space with  the space $\mathcal E$ of entire functions of the form
\begin{equation}\label{eq:ent}
u(z)=\sum_{k=0}^\infty g_kP_k(z),\quad g\in\ell^2, \;z\in\C.
\end{equation}  
By Parseval's Theorem 
$$
\int_{-\infty}^\infty u(x)\overline{v(x)}\,dM(x)=\sum_{k=0}^\infty g_k\overline{h_k},\quad
u=\sum_{k=0}^\infty g_kP_k,\; v=\sum_{k=0}^\infty h_kP_k,\quad g,h\in\ell^2.
$$ 
Note that we have the orthogonal decomposition
\begin{equation}\label{eq:ortdec}
L^2(M)=\mathcal E \oplus \C[x]^\perp,\quad M\in V.
\end{equation}

\section{Proofs}

{\it Proof of Theorem~\ref{thm:indet}}

 Assume $g^{(n)}\in\mathcal D\to 0$ in $\ell^2$ and that $Ag^{(n)}\to f$
in $L^2(M)$. We have to prove that $f=0$. Clearly $f\in\mathcal E$.

 Since $\mathcal E$ is a reproducing kernel Hilbert space of entire functions, we know that convergence  in the Hilbert norm implies locally uniform convergence  in the complex plane, not only for the functions but also for derivatives of any order. Therefore
$$
g^{(n)}_k=\frac{D^k(Ag^{(n)})(0)}{k!}\to \frac{D^k f(0)}{k!},
$$
so the Taylor series of $f$ vanishes because in particular $g^{(n)}_k\to 0$ for $n\to\infty$ for any fixed $k$. $\quad\square$
 
\medskip
{\it Proof of Theorem~\ref{thm:det}}

Let us for simplicity first consider an N-extremal measure $M$ with mass $c>0$ at 0 and consider $\widetilde{M}:=M-c\delta_0$, which is a discrete determinate measure with unbounded support. A concrete example is studied in \cite[p. 128]{B:S1}. The measure $\widetilde{M}$ does not satisfy condition (iii) of Theorem~\ref{thm:Yafaev}. Let $A$ and $\widetilde{A}$ denote the operators \eqref{eq:pol}  with values in $L^2(M)$ and $L^2(\widetilde{M})$ respectively. 
We know that the operator $A:\mathcal D\to L^2(M)$ is closable by Theorem~\ref{thm:indet}.

 Assume that $g^{(n)}\to 0$ in $\ell^2$, where $g^{(n)}\in\mathcal D$, and that $\widetilde{A}g^{(n)}\to f$ in $L^2(\widetilde{M})$. We have $\widetilde{A}g^{(n)}(0)=g^{(n)}_0\to 0$, and therefore
$$
Ag^{(n)}(x)\to \begin{cases} f(x),\quad  & x\in \supp(\widetilde{M})\\
0,\quad &x=0
\end{cases}
$$
in $L^2(M)$ because $M=\widetilde{M}+c\delta_0$. Since $A$ is closable, we conclude that $f=0$.

Let us next modify the proof just given by removing one or finitely many
masses one by one at mass-points $\la_0$ satisfying $|\la_0|<1$ of an N-extremal measure $M$. In fact, for $n\to\infty$ also
$$
\widetilde{A}g^{(n)}(\la_0)=\sum_{k\ge 0} g^{(n)}_k\la_0^k \to 0,
$$
because
$$
|\sum_{k\ge 0} g^{(n)}_k\la_0^k|\le ||g^{(n)}||_{\ell^2} \left(1-|\la_0|^2\right)^{-1/2}.
$$

We finally  claim that if $M$ is  an arbitrary determinate measure with $\ind(M)=n\ge 0$, then the corresponding operator $A$ is closable.  In fact let $\Lambda\subset (-1,1)$ denote a set of $n+1$ points disjoint with $\supp(M)$. Such a choice is clearly possible since the support is discrete in $\R$. By \cite[Theorem 3.9]{B:D} the measure
$$
M^+:=M+\sum_{\la\in\Lambda}\delta_\la
$$
is N-extremal and the corresponding operator $A^+$ is closable by Theorem~\ref{thm:indet}. By removing the masses $\delta_\la$ for $\la\in\Lambda$ one by one we obtain that the operator $A$ associated with $M$ is closable.
 $\quad\square$

\medskip
{\it Proof of Theorem~\ref{thm:strongdet}}

Yafaev's proof is based on a study of the set $\mathcal D_*\subset L^2(M)$ for an arbitrary positive measure $M$ with moments of any order as in \eqref{eq:mom}, namely
\begin{equation}\label{eq:D*}
\mathcal  D_*:=\left\{u\in L^2(M) : u_n:=\int_{-\infty}^\infty u(t) t^n\,dM(t) \in \ell^2\right\}.
\end{equation}
 
Lemma 2.2 in \cite{Y} states that the adjoint $A^*$ of the  operator $A$ from \eqref{eq:pol} is given by $\dom(A^*)=\mathcal D_*$ and 
\begin{equation}\label{eq:A*}
(A^*u)_n=\int_{-\infty}^\infty u(t) t^n\,dM(t),\;n=0,1,\ldots,\quad u\in\mathcal D_*.
\end{equation} 
 
Yafaev uses the following result, Theorem 2.3 in \cite{Y}, which is not  true:

\medskip
{\bf Claim} The following conditions are equivalent:
\begin{enumerate}
\item[(iii)] of Theorem~\ref{thm:Yafaev},
\item[(iv)]  $\mathcal D_*$ is dense in $L^2(M)$.
\end{enumerate}

While it is correct that (iii) implies (iv), the converse is not true. In Theorem~\ref{thm:2} we prove that (iv) holds, if $M$ is an indeterminate  measure and hence (iii) does not hold.
 
For $u\in L^2(M)$ we consider the complex Fourier transform
\begin{equation}\label{eq:Fou}
f(z)=\int_{-\infty}^\infty e^{izt}u(t)\,dM(t),\quad z=x+iy\in\C,
\end{equation}
which is an entire function under the assumption $\root{2n}\of{q_{2n}}=o(n)$. In fact,
\begin{eqnarray*}
|f(z)| &\le& \int_{-\infty}^\infty e^{|t||y|}|u(t)|\,dM(t)=\sum_{n=0}^\infty \frac{|y|^n}{n!}\int_{-\infty}^\infty |t|^n|u(t)|\,dM(t)\\
&\le& \sum_{n=0}^\infty \frac{|y|^n}{n!} \sqrt{q_{2n}}||u||_{L^2(M)}<\infty,
\end{eqnarray*}
because $(\sqrt{q_{2n}}/n!)^{1/n}\to 0$ by Stirling's formula  and the assumption on the moments. In particular $\root{2n}\of{q_{2n}}\le Kn$ for a  suitable constant, and therefore the Carleman condition
$$
\sum_{n=0}^\infty \frac{1}{\root{2n}\of{q_{2n}}}=\infty
$$
secures that the moment problem is determinate, cf. \cite{Ak}.

The function $f$ is considered in \cite[formula (2.6)]{Y} as a $C^{\infty}$-function on the real line, and it is claimed  that it is equal to its Taylor series. This need not be the case under the assumptions in \cite{Y}, but holds true in the present case. Therefore the argument in Yafaev's paper can be carried through. $\quad\square$

\section{Additional results}

We use the following notation  for the orthonormal polynomials \eqref{eq:orpo}.
\begin{eqnarray}
P_n(x)&=&b_{n,n}x^n+b_{n-1,n}x^{n-1}+\ldots + b_{1,n}x+b_{0,n},
\label{eq:p}\\
x^n&=&c_{n,n}P_n(x)+c_{n-1,n}P_{n-1}(x)+\ldots + c_{1,n}P_1(x)+c_{0,n}P_0(x).\label{eq:x}
\end{eqnarray}

The matrices $\mathcal B=\{b_{i,j}\}$ and $\mathcal C=\{c_{i,j}\}$ with the assumption
\[
b_{i,j}=c_{i,j}=0\qquad{\rm for} \ i>j
\]
 are upper-triangular.
Since $\mathcal B$ and $\mathcal C$ are transition matrices between two sequences of linearly independent systems of functions, we have
\begin{equation}\label{eq:BC}
\mathcal B\mathcal C=\mathcal C\mathcal B=\mathcal I.
\end{equation}

Both matrices define operators in $\ell^2$  with domain $\mathcal D$ by defining the image of $e_n\in\mathcal D$ to be the $n$'th column of the matrix. We use the same symbol for these operators as their matrices.
   
In the following we assume the moment problem \eqref{eq:mom} to be indeterminate. In this case $\mathcal B$ extends to a bounded operator on $\ell^2$ which is Hilbert-Schmidt by \cite[Proposition 4.2]{B:S1}. We denote it here $\overline{{\mathcal B}}$, since it is the closure of $\mathcal B$. We know that $\overline{{\mathcal B}}$ is one-to-one by \cite[Proposition 4.3]{B:S1},
and then it is easy to see that $\mathcal C$ is closable and
\begin{equation}\label{eq:oC}
\dom(\overline{\mathcal C})=\overline{\mathcal B}(\ell^2),\quad\overline{\mathcal C}=\overline{\mathcal B}^{-1}.
\end{equation}

\begin{thm}\label{thm:2} Suppose  $M$ is indeterminate. Then the set $\mathcal D _*$ is dense in $L^2(M)$.
\end{thm}  

\begin{proof} 
For $u\in \C[x]^\perp$ we have
$$
u_n = \int_{-\infty}^\infty u(x) x^n\,dM(x)=0,\quad n=0,1,\ldots,
$$
and for $u\in \mathcal E$  given by \eqref{eq:ent} we find
\begin{eqnarray*}
u_n &=& \int_{-\infty}^\infty u(x) x^n\,dM(x)=\sum_{k=0}^\infty g_k\int_{-\infty}^\infty P_k(x)x^n\,dM(x)\\
&=&\sum_{k=0}^n c_{k,n}g_k = (\mathcal C^t g)_n,
\end{eqnarray*}
where we have used \eqref{eq:x}.

By the orthogonal decomposition \eqref{eq:ortdec} we find
$$
\mathcal D_*=\left\{u=\sum_{k=0}^\infty g_kP_k \mid g\in\ell^2,\;\mathcal C^t g\in\ell^2\right\}\oplus \C[x]^\perp, 
$$
so $\mathcal D_*$ is dense in $L^2(M)$ if and only if
$$
 X:=\{g\in\ell^2 \mid \mathcal C^t g\in\ell^2\} \;\mbox{ is dense in }\; \ell^2.
$$

However, $\{\mathcal B^t\eta \mid \eta\in\mathcal D\}\subset X$ and the subset is already dense in $\ell^2$.

In fact, for $\eta\in\mathcal D$ we  have $\mathcal B^t \eta\in\ell^2$ because the matrix $\mathcal B$ is Hilbert-Schmidt.
 Furthermore, $\mathcal C^t(\mathcal B^t \eta)=\eta\in \ell^2$ because of \eqref{eq:BC}.

Finally, since $\overline{\mathcal B}$ is a bounded operator and one-to-one on $\ell^2$,  the set $\{\mathcal B^t\eta\, \mid \,\eta\in \mathcal D\}$ is dense in $\ell^2$.
\end{proof}

By Theorem~\ref{thm:indet} we know that the operator $A$ given by \eqref{eq:pol} is closable, when $M$ is indeterminate. We shall now describe the closure $\overline{A}$ in this case. For this we need the unitary operator 
$U:\ell^2\to\mathcal E$ given by $U(e_n)=P_n,\,n=0,1,\ldots$.

\begin{thm}\label{thm:oA} Suppose  $M$ is indeterminate. Then 
\begin{equation}\label{eq:oA}
\dom(\overline{A})=\overline{\mathcal B}(\ell^2), \quad\overline{A}=U\overline{\mathcal C}.
\end{equation}
For $\xi\in\dom(\overline{A})$ we have $\xi=\overline{\mathcal B}y$ for a unique $y\in\ell^2$ and the following series expansions hold
\begin{equation}\label{eq:oA1}
\overline{A}\xi(z)=\sum_{k=0}^\infty \xi_k z^k=\sum_{n=0}^\infty y_nP_n(z),\quad z\in\C,
 \end{equation}
uniformly for $z$ in compact subsets of $\C$.
\end{thm}

\begin{proof} We clearly have $A=U\mathcal C$, hence $\overline{A}=U\overline{\mathcal C}$, and therefore $\dom(\overline{A})=\dom(\overline{\mathcal C})=\overline{\mathcal B}(\ell^2)$.

For $\xi=\overline{\mathcal B}y$ for $y\in\ell^2$, we have $\overline{A}\xi=Uy$ and
$$
f(z):=Uy(z)=\sum_{n=0}^\infty y_nP_n(z),
$$
uniformly for $z$ in compact subsets of $\C$. By Cauchy's integral formula we therefore get
\begin{eqnarray*}
\frac{f^{(k)}(0)}{k!}&=&\frac{1}{2\pi i}\int_{|z|=1}\frac{f(z)}{z^{k+1}}\,dz
=\sum_{n=0}^\infty y_n\frac{1}{2\pi i}\int_{|z|=1}\frac{P_n(z)}{z^{k+1}}\,dz\\
&=&\sum_{n=0}^\infty y_nb_{k,n}=\xi_k.
\end{eqnarray*}
This shows the first expression in \eqref{eq:oA1}.
\end{proof}

We end with  an example related to Yafaev's condition (iii).

\begin{ex} Let $M$ be  a positive measure on $[-1,1]$ with $M(\{1\})=c>0$. The operator $A$ is not closable.
\end{ex}

In fact, define 
$$
g^{(n)}_k=\begin{cases} 1/n, \quad & 0\le k\le n-1,\\
0, \quad & k\ge n.
\end{cases}
$$
Then $g^{(n)}\to  0$ in $\ell^2$. We have, 
$$
Ag^{(n)}(x)=\begin{cases} 1, \quad & x=1,\\
\frac{1}{n}\frac{1-x^n}{1-x}, \quad & -1\le x<1.
\end{cases}
$$
Hence $Ag^{(n)}(x)\to \chi_{1}(x)$ pointwise and also in $L^2(M)$, where $\chi_B$ denotes the indicator function of a subset $B$ of the real line. Thus $A$ is not closable.

\noindent
Christian Berg\\
Department of Mathematical Sciences, University of Copenhagen\\
Universitetsparken 5, DK-2100 Copenhagen, Denmark\\
e-mail: {\tt{berg@math.ku.dk}}

\vspace{0.4cm}
\noindent
Ryszard Szwarc\\
Institute of Mathematics, University of Wroc{\l}aw\\
pl.\ Grunwaldzki 2/4, 50-384 Wroc{\l}aw, Poland\\ 
e-mail: {\tt{szwarc2@gmail.com}}

\end{document}